\documentclass{article}
\usepackage{amssymb}
\usepackage{amsmath}
\usepackage[normalem]{ulem}

\newtheorem{theorem}{Theorem}[section]
\newtheorem{lemma}[theorem]{Lemma}
\newtheorem{proposition}[theorem]{Proposition}
\newtheorem{corollary}[theorem]{Corollary}

\newtheorem{preexample}{Example}[section]

\newtheorem{preremark}{Remark}
\newenvironment{remark}{\begin{preremark}\rm}{\end{preremark}}

\newenvironment{proof}
{{\bf Proof:}}
{\qquad \hspace*{\fill} $\Box$}%

\newcommand{\fa}{\mathfrak{a}}%
\newcommand{\fg}{\mathfrak{g}}%
\newcommand{\fk}{\mathfrak{k}}%
\newcommand{\fn}{\mathfrak{n}}%

\newcommand{\Ad}{\operatorname{Ad}}%
\newcommand{\ad}{\operatorname{ad}}%
\newcommand{\Sl}{\operatorname{Sl}}%
\newcommand{\DC}{\mathcal{D}}%
\newcommand{\XC}{\mathcal{X}}%
\newcommand{\R}{\mathbb{R}}%
\newcommand{\Z}{\mathbb{Z}}%

\begin{document}

\title{Periodic orbits of Linear or invariant flows on connected Lie groups}%: the solution, stability, conjugacy, and an application on the 3-dimensional case}

\author{S. N. Stelmastchuk\\ Universidade Federal do Paran\'{a}\\Jandaia do Sul, Brazil
,
\footnote{{\bf e-mail: }simnaos@gmail.com}
}

\maketitle

\begin{abstract}
	Our main is to study periodic orbits of linear or invariant flows on a real, connected Lie group. Since each linear flow $\varphi_t$ has a derivation associated $\DC$, we show that the existence of periodic orbits of $\varphi_t$ is based on the eigenvalues of the derivation $\DC$. From this, we study periodic orbits of a linear flow on noncompact, semisimple Lie groups, and we work with periodic orbits of a linear flow on connected, simply connected, solvable Lie groups of dimension 2 or 3. 
\end{abstract}



\section{Introduction}

The study of periodic orbits is a recurrent issue. In sprayed literature, many theory and techniques are presented to find and classify periodic orbits. For a fuller treatment the reader we refer the reader to  \cite{chicone}, \cite{colonius}, and \cite{robinson} among others. Our main is to study periodic orbits of the linear and invariant flows on real, connected Lie groups. Following we introduce our approach.

Let $G$ be a real, connected Lie group. Let us denote by $\fg$ its Lie algebra. We recall that a vector field $\XC$ on $G$ is called linear if its flow $(\varphi_t)_{t \in \R}$ is a family of automorphism of the Lie group $G$. Namely, the linear flow  $\varphi_t$ is the solution of dynamical system 
\begin{equation}\label{odelinear}
	\dot{g} = \XC(g),\, g \in G.
\end{equation}

In \cite{stelmastchuk}, the author study periodic orbits of a linear flow on a real, connected, compact, semisimple Lie group. Our wish is to extend the study of periodic orbits of a linear flow (\ref{odelinear}) to a real, connected Lie group $G$. 

An important fact is that the linear vector field $\XC$ yields the derivation $\DC = - \ad(\XC)$ of the Lie algebra $\fg$. Our idea is to associated the eigenvalues of $\DC$ with existence of periodic orbits. This idea comes from the study of periodic orbits of linear dynamical system on Euclidian spaces (see for instance \cite{colonius}). Here, care is needed because for every linear vector field $\XC$ we have a derivation $\DC$ associated, however the contrary it is not true (see \cite{sanmartin} for more detail). 

Our main result is that an orbit that is not a fixed point of a linear flow $\varphi_t$ is periodic if the eigenvalues of the derivation $\DC$ which are not null are $\pm \alpha_1 \operatorname{i}, \ldots, \pm \alpha_r \operatorname{i}$ with rational quotient $\alpha_i/\alpha_j$ for $i,j =1,\ldots r$, and if they are semisimple. Let $X$ be a right invariant vector field on $G$. It is possible to associate a linear vector field $\XC$ for $X$. Thus, we show that the invariant flow $\exp(tX)$ has periodic orbits if and only if the eigenvalues of the derivation $\DC=- \ad(X)$ which are not null are $\pm \alpha_1 \operatorname{i}, \ldots, \pm \alpha_r \operatorname{i}$ and semisimple with rational quotient $\alpha_i/\alpha_j$ for $i,j =1,\ldots r$.

In \cite{stelmastchuk}, the author studied the periodic orbits of linear flows on a compact, semisimple Lie group. Using the main result we study periodic orbits of linear flows in a noncompact, semisimple Lie group $G$, where we use the Iwasawa's decomposition $\fg = \fk \oplus \fa \oplus \fn$. In fact, we show that orbits that are not fixed points of the linear flow $\varphi_t$ with derivations $\DC = - \ad(aH +bX)$ are not periodic, where $H \in \fa, X \in \fn$, and $a, b \in \R$. Since $\fk$ is the compact part, a natural question arises: is the unique possibility of linear flow has periodic orbit for derivations $\DC= -\ad(Y)$ with $Y \in \fk$? No. To view this, we study periodic orbits of a linear flow on the special linear group $\Sl(2,\R)$. 

Finally, we work with real, connected, simply connected, solvable Lie groups of dimension 2 or 3, following the classification presented in \cite{biggs}.

This paper is organized as follows. Section 2 briefly reviews the notions of linear vector fields. Section 3 works with periodic orbits of linear and invariant flows on a connected Lie group. Section 4 studies periodic orbits on noncompact, semisimple Lie groups. Section 5 works with periodic orbits of a linear flow $\varphi_t$ on a connected, simply connected, solvable Lie group of dimension 2 or 3. 

\section{Linear vector fields}

In this section we recall some basic facts about linear vector field. For a fuller treatment we refer to \cite{ayalatirao}, \cite{cardetti}, and \cite{jouan}. Let $G$ be a connected Lie group, and let us denote by $\fg$ its Lie algebra. A vector field $\XC$ on $G$ is called linear if its flow $(\varphi_t)_{t\in\R}$ is a family of automorphisms of the Lie group $G$. For a linear vector field $\XC$ it is possible to associated a derivation of Lie algebra $\fg$:
\[
	\DC(Y) = - [\XC,Y], \, Y \in \fg. 
\] 

For the convenience of the reader we resume some facts about a linear vector field $\XC$ and its flow $\varphi_t$. The proof of these facts can be found in \cite{jouan}. 
\begin{proposition}\label{linearproperties}
	Let $\XC$ be a linear vector field and $\varphi_t$ its flow. The following assertions are equivalent:
	\begin{description}
		\item{(i)} the linear flow $\varphi_t$ is an automorphism of Lie groups for each $t$;
		\item{(ii)} for $g,h \in G$ it follows that $\XC$ is linear iff $\XC(gh) = R_{h*}\XC(g) + L_{g*}\XC(h)$; 
		\item{(iii)} at identity $e$, we have $(d\varphi_t)_e = e^{t\DC}$ for all $t \in \R$.
	\end{description}
\end{proposition} 

It is known that $G= \R^n$ is a Euclidean Lie Group. For any $n\times n$ matrix $A$ it is true that $\XC = A$ is a linear vector field. Furthermore, $\DC_x(b) = -[Ax,b] = Ab$. In this sense, we can view the dynamical system 
\[
	\dot{g} = \XC(g), \ \ g \in G,
\]
as a generalization of dynamical system on $\R^{n}$ given by
\[
	\dot{x} = Ax, \ \ x \in \R^n.
\]

%\begin{proposition}\label{derivationlinearflow}
%	If $\varphi_t$ is a linear flow, then $(d\varphi_t)_g = (dR_{\varphi_t(g)}) \circ e^{t\DC}$ for any $g \in G$.
%\end{proposition}
%\begin{proof}
%	It is sufficient to prove for $X \in \fg$. Let us calculate
%	\begin{eqnarray*}
%		(d\varphi_t)(g)(X) = \left.\frac{d}{ds}\right|_{s=0} \!\!\!\varphi_t(\exp(sX)g) = \left.\frac{d}{ds}\right|_{s=0} \!\!\!\exp(se^{t\DC}X)\varphi_t(g) = (dR_{\varphi_t(g)})(e^{t\DC}X),
%	\end{eqnarray*}
%	which establishes the formula.
%\end{proof}

\section{Periodic orbits}

In this section we assume that $G$ is a connected Lie group. Before we prove our main theorem, we need to introduce some concepts. Following \cite{colonius}, if for an eigenvalue $\mu$ all complex Jordan blocks are one-dimensional, i.e., a complete set of eigenvectors exists, it is called semisimple. Equivalently, the corresponding real Jordan blocks are one-dimensional if $\mu$ is real, and two-dimensional if $\mu$ and $\bar{\mu} \in \mathbb{C} \setminus \mathbb{R}$. 

\begin{proposition}\label{prop1}
	Let $\fg$ be a Lie algebra. If $\DC$ is a derivation of $\fg$, then the following conditions are equivalent:
	\begin{description}
		\item[i)] the flow $e^{t\DC}$ is periodic; 
		\item[ii)] the eigenvalues of the derivation $\DC$ which are not null are $\pm \alpha_1 \operatorname{i}, \ldots, \pm \alpha_r \operatorname{i}$ and semisimple with rational quotient $\alpha_i/\alpha_j$ for $i,j =1,\ldots r$.
	\end{description}
\end{proposition}
\begin{proof}
	We first suppose that the flow $e^{t\DC}$ is periodic. It means that there exists $T>0$ such that  
	\[
		e^{(t+T)\DC} = e^{t\DC}.
	\] 
	A simple calculus shows that $e^{T\DC} = Id$. If $J$ is the Jordan form of $\DC$, then $e^{TJ} =Id$. 
	
	We break the proof in two steps: \\
	i) real eigenvalues: Let $\mu$ be a real eigenvalue of derivation $\DC$. Let us denote by $J_\mu$ the m-dimensional Jordan block of $\mu$. We thus get $e^{TJ_\mu} =I_m$. Consider $m>1$. From Jordan block $J_\mu$ we see that $e^{T\mu}(T) =0$. It implies that $T=0$, a contradiction. Therefore $m=1$. It means that $\mu$ is a semisimple eigenvalue. Consequently, $e^{TJ_\mu} =I_1$, which gives $e^{T\mu} = 1$. We thus get $T \mu = 0$. Since $T>0$, it follows that $\mu =0$. It shows that unique real eigenvalue  of derivation $\DC$ is $0$. \\
	ii) complex eigenvalues: suppose that $\mu = \alpha \pm i\beta$ are conjugate, complex eigenvalues of derivation $\DC$. Write
	\[
		R = R(t) = \left(
	  \begin{array}{cc}
			\cos(t\beta) & -\sin(t\beta) \\
			\sin(t\beta) & \cos(t\beta)
		\end{array}
		\right).
	\]
	Let $J_\mu$ denote the 2m-dimensional Jordan block of $\mu$. Suppose that $m>1$. From $e^{TJ_{\mu}} = I_{2m}$ we see that 
	\[
		e^{\alpha T}\cdot T\cdot R = 
	\left(
	  \begin{array}{cc}
			0 & 0 \\
			0 & 0
		\end{array}
		\right).
	\]
	A simple account shows that $T=0$, a contradiction. Therefore $m=1$. It means that $\mu$ is semisimple. Hence the Jordan block  $J_\mu$ is two dimensional. It gives $e^{TJ_{\mu}} = I_{2}$. From this equality we conclude that
	\[
		e^{\alpha T} \cos(\beta T) = 1 \ \ \mbox{and} \ \ e^{\alpha T} \sin(\beta T) = 0.
	\]
	The second equality implies that $\sin(\beta T) = 0$. It follows that $\beta T = n \pi$ for any $n \in \Z$. Substituting this in the first equation gives 
	\[
		e^{\alpha T} \cos(n \pi) = 1.
	\]
	Because $e^{\alpha T}>0$, it follows that $n = 2m$ for some $m \in Z$. Hence $e^{\alpha T} = 1$. So $\alpha T = 0$. Since $T>0$, we have $\alpha = 0$. It means that complex eigenvalues of derivation $\DC$ are of the form $\mu = \pm \beta i$. 

	Now we proved that nonnull complex eigenvalues $\pm \alpha \operatorname{i}$ of $\DC$ yield rational quotient. As proved above, any real Jordan Block of $\DC$ has dimension 1 or 2 if it is real or complex, respectively. Furthermore, we showed that unique real eigenvalue of $\DC$ is $0$. Therefore its real Jordan block is written as $J_0 =[0]$. Thus $e^{tJ_0}$ is constant. It means that in direction of $0$ the curve $e^{tJ}$ is constant. Consequently, solutions associated to the eigenvalue $0$ are trivially periodic. On the other hand, suppose that there exist non null complex eigenvalues $\pm\alpha_i \operatorname{i}$, $i=1, \ldots, r$. By proved above, its real Jordan blocks are  
	\[
		\left(
		\begin{array}{cc}
			\cos(t\alpha_i) & -\sin(t\alpha_i) \\
			\sin(t\alpha_i) & \cos(t\alpha_i)
		\end{array}
		\right), \,\, i= 1, \dots r.
	\]
		As $e^{TJ} =Id$ we have $\alpha_i \cdot T = p_i \cdot 2\pi$ for some $p_i \in \Z$, $i=1, \ldots, r$. It means for any $i,j = 1, \ldots, r$ that $\alpha_i/\alpha_j = p_i/p_j$ is a rational number 
		
	Reciprocally, assume that the eigenvalues of $\DC$ are semsimple and they are $0$ or $\pm\alpha_1 \operatorname{i}, \ldots, \pm \alpha_r \operatorname{i}$ with $\alpha_i \neq 0$, $i=1, \ldots, r$, and  $\alpha_i/\alpha_j$ is rational for $i,j=1,\ldots n$. Trivially the solution is constant for the eigenvalue $0$. We thus work with the eigenvalues $\pm \alpha_i \operatorname{i}$ with $\alpha_i \neq 0$. By assumption, $\pm \alpha_i \operatorname{i}$ is semisimple eigenvalue for $i = 1, \ldots, r$. It implies that every real Jordan block has dimension two and the solution applied at this block gives the following matrix 
	\[
		\left(
	  \begin{array}{cc}
			\cos(t\alpha_i) & -\sin(t\alpha_i) \\
			\sin(t\alpha_i) & \cos(t\alpha_i)
		\end{array}
		\right).
	\]
	On the other hand, we know that there exists $p_{ij}, q_{ij} \in \Z$ with $q_{ij} >0$ such that $\alpha_i/\alpha_j = p_{ij}/q_{ij}$ for $i,j = 1, \ldots r$. In particular, we can written $\alpha_i = (p_{i1}/q_{i1})\alpha_1$ for $i=2, \ldots, r$. Supposing that $\alpha_1> 0$ it is sufficient to take $T= q_{21}q_{31}\ldots q_{r1}  (2\pi/\alpha_1)$ to get satisfies $e^{TJ} =Id$, where $J$ is the Jordan form. In other words, $Id$ is a periodic point of $e^{TJ}$ with period $T>0$, which is equivalent $Id$ to be periodic point of $e^{TD}$. 
\end{proof}

\begin{remark}\label{rem1}
	Let $\XC$ be a linear vector field on $G$. In a natural way, we define a derivation $\DC = - \ad(\XC)$ on the Lie algebra $\fg$ associated to $\XC$. On contrary, it is not true that a derivation yields a linear vector field if $G$ is only connected. However, if $G$ is connected e simply connected, then there is a one-to-one relation between derivation and linear vector field (see \cite{sanmartin} for more details). 
\end{remark}

Let $\tilde{G}$ be the simply connected covering of the Lie group $G$ and $\pi: \tilde{G} \rightarrow G$ the canonical projection. Let $\XC$ be a linear vector field and $ \DC$ its associated derivation. By Theorem 2.2 of \cite{ayalatirao}, there exists a unique linear vector field $\tilde{\XC}$ on $\tilde{G}$ whose associated derivation is $\DC$. Let us denote by $\tilde{\varphi}_t$ the flow of $\tilde{\XC}$. Thus for any $X \in \fg$ we have  
\[
	\pi(\tilde{\varphi}_t(\exp_{\tilde{G}} X)) = \pi(\exp_{\tilde{G}}(e^{t\DC} X)) = \exp_{G}(e^{t\DC} X) =\varphi_t(\exp_{G} X).
\]
By connectedness, it holds that
\[
	\pi \circ \tilde{\varphi}_t = \varphi_t \circ \pi \ \ \mbox{ for any} \ \ t \in \R
\]
implying that $\tilde{\XC}$ and $\XC$ are $\pi$-related. 

\begin{lemma}\label{lem1}
	Under the conditions stated above, if $g \in G$ and $\tilde{g} \in \tilde{G}$ such that $\pi(\tilde{g}) =g$ and if flow $\tilde{\varphi}_t(\tilde{g})$ is periodic with period $T$, then also is $\varphi_t(g)$.  
\end{lemma}
\begin{proof}
	Suppose that $\tilde{\varphi}_t(\tilde{g})$ is periodic. Then there exists a time $T>0$ such that $\tilde{\varphi}_{t+T}(\tilde{g}) = \tilde{\varphi}_t(\tilde{g})$. Since $\pi(\tilde{g}) = g$, it follows 
	\[
		\varphi_{t+T}(g) = \varphi_{t+T}(\pi(\tilde{g})) = \pi(\tilde{\varphi}_{t+T}(\tilde{g})) = \pi(\tilde{\varphi}_{t}(\tilde{g})) = \varphi_{t}(\pi(\tilde{g}))= \varphi_{t}(g).
	\]
	Therefore $\varphi_t(g)$ is a periodic orbit with period $T>0$.
\end{proof}

\begin{theorem}\label{teo1}
	Let $G$ be a connected Lie group. Let $\XC$ be a linear vector field on $G$ and denote by $\DC$ and $\varphi_t$ their derivation and flow, respectively. If the eigenvalues of the derivation $\DC$ which are not null are $\pm \alpha_1 \operatorname{i}, \ldots, \pm \alpha_r \operatorname{i}$  and semisimple with rational quotient $\alpha_i/\alpha_j$ for $i,j =1,\ldots r$, then there exist periodic orbits for the linear flow $\varphi_t$.
\end{theorem}
\begin{proof}
	Suppose that our assumptions about eigenvalues of the derivation $\DC$ are true. Therefore the flow $e^{t\DC}$ is periodic by Proposition \ref{prop1}. It means that there is a $T>0$ such that $e^{(t+T)\DC} = e^{t\DC}$ for any $t \in \R$. By Proposition \ref{linearproperties}, 
	\[
		(d\tilde{\varphi}_{t+T})_e = (d \tilde{\varphi}_t)_e.
	\]
	Since $\tilde{G}$ is connected and $\tilde{\varphi}_{t+T}$  and $\tilde{\varphi}_t$ are homomorphism, it follows that $\tilde{\varphi}_{t+T}(\tilde{g}) =  \tilde{\varphi}_{t}(\tilde{g})$ for any $\tilde{g} \in \tilde{G}$. It means that $\tilde{\varphi}_t(\tilde{g})$ is a periodic orbit. Take $g \in G$ such that $\pi(\tilde{g})= g$. By Lemma \ref{lem1}, $\varphi_t(g)$ is periodic orbit with period $T$ if $g$ is not a fixed point. 
\end{proof}

 A direct consequence is that any linear flow $\varphi_t$ with derivation $\DC$ that have eigenvalues with real part nonnull do not have periodic orbits.

\begin{corollary}\label{cor1}
	Under the hypotheses of Theorem \ref{teo1}, if a derivation $\DC$ on $\fg$ has only real eigenvalues, then there are not periodic orbits to the linear flow $\varphi_t$ associate to $\DC$.
\end{corollary}
\begin{proof}
	Suppose, contrary our claim, that $\varphi_t(g)$ is a periodic orbit for a not fixed point $g \in G$ . From Theorem \ref{teo1} we see that the unique real eigenvalue is $0$. Furthermore, $0$ is a semisimple eigenvalue. It implies that $\DC$ is null. By Proposition \ref{linearproperties},
	\[
		(d\varphi_t)_{e} = e^{tD} = I_\fg,\, t \in \R, 
	\]
	where $I_\fg$ is the identity map. Denoting by $I_G$ the identity homomorphism of $G$ we see that $\varphi_t = I_G$ for any $t$ because $G$ is connected. This gives $\varphi_t(g) =g$ for every $t \in \R$. It means that every $g$ is a fixed point, which is a contradiction. 
\end{proof}

In the sequel, let $X$ be a right invariant vector field on $G$. Define a vector field by $\XC = X + I_*X$, where $I_{*}X$ is the left invariant vector field associated to $X$. Here $I_*$ is the differential of inverse map $\mathfrak{i}(g) = g^{-1}$ (more details is founded in \cite{sanmartin}). It is possible to show that $\XC$ is linear and its associated derivation is given by $\DC = -\ad(\XC) = - \ad(X)$. Furthermore, the differential equation (\ref{odelinear}) is written as 
\[
	\dot{g} = X(g) + (I_*X)(g), \,\, g \in G.
\]

It is possible to show that linear flow $\varphi_t$ is solution of (\ref{odelinear}) if and only if $\varphi_t(g)\cdot \exp(tX)$ is solution of $\dot{g} = X(g)$. 

\begin{theorem}\label{teo2}
	Let $G$ be a connected Lie group and $X$ be a right invariant vector field on $G$. The following conditions are equivalent:
	\begin{description}
		\item[i)] there exists a periodic orbit for the right invariant flow $\exp(tX)$; 
		\item[ii)] the eigenvalues of the derivation $\DC= -\ad(X)$ which are not null are $\pm \alpha_1 \operatorname{i}, \ldots, \pm \alpha_r \operatorname{i}$ and semisimple with rational quotient $\alpha_i/\alpha_j$ for $i,j =1,\ldots r$.
	\end{description}
\end{theorem}
\begin{proof}
	Assume that there exists a $g \in G$ such that $\exp(tX)g$ is periodic with period $T>0$. It is equivalent to $\exp(tX)$ be periodic with period $T>0$. From this we deduce that 
	\begin{eqnarray*}
		\exp((t+T)X) = \exp(tX) 
		& \Rightarrow & \Ad(\exp(-(t+T)X)) = \Ad(\exp(-tX)) \\
		& \Rightarrow & e^{(t+T)\DC} = e^{t\DC},
	\end{eqnarray*}
	which means that the flow $e^{t\DC}$ is periodic. By Proposition \ref{prop1}, the eigenvalues of the derivation $\DC$ are the form $0$ or $\pm \alpha_1 \operatorname{i}, \ldots, \pm \alpha_r \operatorname{i}$ where $\alpha_i \neq 0$, $i=1, \ldots,r$, $\alpha_i/\alpha_j$ is a rational for $i,j =1,\ldots r$, and $\pm \alpha_i \operatorname{i}$, $i=1, \ldots, r$, are semisimple.
		
	Reciprocally, suppose that sentece ii) is true. By Proposition \ref{prop1}, there exist a time $T>0$ such that $e^{(t+T)\DC} = e^{tD}$. We thus get $\exp((t+T) X) = \exp(tX)$, which means that $\exp(tX)$ is periodic with period $T>0$. 
\end{proof}

Let $\DC$ be a inner derivation on $G$. We know that there exists a right invariant vector field such that $\DC = - \ad(X)$. Thus we can construct the vector field $\XC = X + I*X$ as above, which is linear and associated to $\DC$. It means that every inner derivation yields a linear vector field independent if $G$ is simply connected or not. If $G$ is semisimple Lie group, then every derivation is inner. 

\section{Semisimple Lie group}

In \cite{stelmastchuk}, the author study periodic orbits on compact, semisimple Lie groups. An important fact in this approach is that $\DC$ has only semisimple eigenvalues, therefore it simplifies the work to find periodic orbits. Now let $G$ be a noncompact, semisimple Lie group. From Iwasawa's decomposition there exists three Lie subalgebra $\fk, \fa$, and $\fn$ such that $\fg = \fk \oplus \fa \oplus \fn$. One may think that periodic orbits are in compact part of $G$ and that there are not periodic orbits on soluble part of Iwasawa's decomposition. We now work to verify this last assertion.

\begin{proposition}\label{periodicII}
	Let $G$ be a noncompact, semisimple Lie group such that its Lie algebra is decomposed as $\fg = \fk \oplus \fa \oplus \fn$. If $H \in \fa, X \in \fn$ and if $\DC = -\ad(aH+ bX)$ is a derivarion for some $a,b \in \R$, then orbits that are not fixed points of the linear flow associated to $\DC$ are not periodic 
\end{proposition}
\begin{proof}
	Let $\DC = -\ad(aH+ bX)$ be a derivation  for some $a,b \in \R$, $H \in \fa$, and $X \in \fn$. From Lemma 6.4.5 in \cite{knapp} there exists a basis of $\fg$ such that $\ad(H)$ is a diagonal matrix with real entries and $\ad(X)$ is a upper triangular matrix with 0's on the diagonal. Hence the matrix of $\DC$ is a upper triangular matrix with real entries on diagonal. In consequence, all eigenvalues of $\DC$ are real. From Corollary \ref{cor1} it follows that orbits that are  not fixed points of the linear flow associated to $\DC$ are not periodic.  
\end{proof}

Proposition above says that linear flow that have only contribution of $\fa \oplus \fn$ are not periodic. In this way, if there exists periodic orbit for linear flow on $G$, then the derivation $\DC$ must have contribution of elements of $\fk$. To verify this, we study periodic orbits of linear flow on the semisimple Lie group $\Sl(2,\R)$. Let us denote by $\operatorname{sl}(2,\R)$ the Lie algebra of $\Sl(2,\R)$. Since $\Sl(2,\R)$ is a semisimple non-compact Lie groups, its Lie algebra has the following Iwasawa's decomposition:
\[
	\operatorname{sl}(2,\R) = \left\{\left(
	\begin{array}{cc}
		0 & -\alpha \\
		\alpha & 0
	\end{array}
	\right), \alpha \in \R
	\right\} \oplus
	\left\{
	\left(
	\begin{array}{cc}
		a & 0 \\
		0 & -a
	\end{array}
	\right), a \in \R
	\right\} \oplus
	\left\{
	\left(
	\begin{array}{cc}
		0 & \nu \\
		0 & 0
	\end{array}
	\right), \nu \in \R \right\}.
\] 
It is clear that 
\[
	\beta = \left \{
	Y=\left(
	\begin{array}{cc}
		0 & -1 \\
		1 & 0
	\end{array}
	\right),
	H =\left(
	\begin{array}{cc}
		1 & 0 \\
		0 & -1
	\end{array}
	\right),
	Z = \left(
		\begin{array}{cc}
		0 & 1 \\
		0 & 0
	\end{array}
	\right)
	\right\}
\]
form a basis of $\operatorname{sl}(2,\R)$. Furthermore, they brackets are given by
\[
	[Y,H] = 2YX + 4Z , \, [Y,Z] = - H, \, [H, Z] = 2Z.
\]
Let $\XC$ be a linear vector field on $\Sl(2,\R)$ and $\DC$ its derivation. Then there exist a right invariant vector field $X \in \operatorname{sl}(2, \R)$ such that $\DC = - \ad(X)$. Write $X$ as 
\[
	X = aY + bH + cZ,\, a, b, c \in \R.
\]
In the basis $\beta$, the matrix of derivation is written as 
\[
	\DC= - \ad(X) = \left(
	\begin{array}{ccc}
		2 b & -2 a & 0 \\
		-c & 0 & a \\
		4 b & -4 a+2 c & -2 b
	\end{array}
	\right).
\]
From this we see that the eigenvalues of $\DC$ are 
\begin{equation}\label{eigenvalues}
	\left\{0,\,-2 \sqrt{-a^2+a c+b^2},\,2 \sqrt{-a^2+a c+b^2}\right\}.
\end{equation}
We are in a position to show a condition to an orbit of linear flow  be periodic.
\begin{proposition}\label{prop2}
	Let $\XC$ be a linear vector field on $\Sl(2,\R)$ and $X$ its associated right invariant vector field. Write $X = aY + bH + cZ$ in the basis $\beta$ of $\operatorname{sl}(2,\R)$. Orbits of the linear flow $\varphi_t$ of $\XC$ that are not fixed points are periodic if $a^2>ac+b^2$.
\end{proposition}
\begin{proof}
	It is a direct application of Theorem \ref{teo1} with eigenvalues (\ref{eigenvalues}).
\end{proof}

%\begin{proposition}\label{prop3}
%	Let $X$ be an invariant vector field on $\Sl(2,\R)$. Write $X = aY + bH + cZ$ in the basis $\beta$ of $\operatorname{sl}(2,\R)$. A necessary and sufficient condition to the orbits of the invariant flow $\varphi_t$ of $\XC$ be periodic is $a^2>ac+b^2$.
%\end{proposition}
%\begin{proof}
%	It is a direct application of Theorem \ref{teo2} with eigenvalues (\ref{eigenvalues}).
%\end{proof}

The meaning of Propositions above is that to study periodic orbits of invariant and linear flows on $\Sl(2,\R)$ are necessary to consider the compact, abelian, and nilpotent parts of Iwasawa's decomposition. In other words, we need to consider a derivation $\DC = - \ad(X)$ with $X= H + A +N$ where $H \in \fk, A \in \fa$ and $N \in \fn$.

\section{Periodic orbits for linear flow on low dimension Lie group}

In this section we study periodic orbits of linear flow on connected, simply connected, and solvable Lie groups of dimension 2 or 3. In the semisimple case, $SL(2,\mathbb{R})$ was studied in the section above, and the unitary group $SU (2)$ and the orthogonal group $SO(3)$ were studied in \cite{stelmastchuk}.

%Following Remark \ref{rem1}, in the case soluble we only consider the connected and simply connected Lie groups. Already in the semisimple Lie groups every derivation is inner. It implies that any inner derivation yields a linear vector field in a standard way.  

\subsection{Dimension 2}
Let $\fg$ be a Lie algebra of dimension 2. It is well know that there exists two possibilities of Lie algebras: abelian and solvable. 

\subsubsection{Abelian case}
In this case any derivation is given by
\[
	\DC = 
	\left(
	\begin{array}{c c}
		a & b \\
		c & d \\
	\end{array}
	\right).
\]
A simple account shows that the eigenvalues are 
\[
	\left\{\frac{1}{2} \left(a+d-\sqrt{(a-d)^2+4bc}\right),\frac{1}{2} \left(a+d+\sqrt{(a-d)^2+4bc}\right)\right\}.
\]
By Theorem \ref{teo1}, the linear flow associated for the derivation $\DC$ has periodic orbits if $a+d=0$ and $(a-d)^2+4bc<0$.

\subsubsection{Solvable case}
The solvable, connected, simply connected Lie group of dimension 2 is the affine group $Aff_0(2)$. In \cite{jouan}, it is showed that derivations are inner and they are given by
\[
	\DC = 
	\left(
	\begin{array}{c c}
		0 & 0 \\
		c & d \\
	\end{array}
	\right).
\]
It is easy to see that eigenvalues are $\{0, d\}$. By Corollary \ref{cor1}, any linear flow in $Aff_0(2)$ do not have periodic orbit.

\subsection{Dimension 3}

Let $G$ be a connected Lie group of dimension 3. In \cite{onishchik} it is classified all connected Lie groups of dimension 3, and in \cite{biggs} we find a clear presentation of this classification. In \cite{biggs}, we find that for an (appropriate) ordered basis $(E1,E2,E3)$ of $\fg$ the Lie bracket of a connected Lie group of dimension 3 is given by 
\begin{equation}
	\label{colchetes}
	\begin{array}{rcl}
		[E_1,E_2] & = & n_3E_3 \\
		\left[E_3,E_1\right] & = & a E_1 + n_2 E_2\\
		\left[E_2,E_3\right] & = & n_1E_1 - aE_2,
	\end{array}
\end{equation}
where $a$, $n_1$, $n_2$ e $n_3$ are given by table 1 in \cite{biggs}. In each case study bellow, we give values of $a$, $n_1$, $n_2$ and $n_3$. 

Let $\DC$ be a derivation on Lie algebra $\fg$. In the basis $(E_1,E_2,E_3)$, we can write the derivation $\DC$ as  
\begin{eqnarray*}
	\DC(E_1) & = & x_1 E_1 + y_1E_2+ z_1 E_2\\
	\DC(E_2) & = & x_2 E_1 + y_2E_2+ z_2 E_2\\
	\DC(E_3) & = & x_3 E_1 + y_3E_2+ z_3 E_2
\end{eqnarray*}
with $x_1,x_2,x_3,y_1,y_2,y_3, z_1, z_2,z_3 \in \R$. Applying Lie Brackets (\ref{colchetes}) at equalities 
\[
	\DC[E_i,E_j] = [\DC(E_i),E_j] + [E_j, \DC(E_i)],\ \ \text{for }\,  i,j= 1, 2,3 \ \ \text{and} \ \ i\neq j, 
\]
we obtain the following equation system
\begin{equation}\label{sys1}
	\left\{
	\begin{array}{ccc}
		n_3x_3 + n_1 z_1 +az_2    & = & 0\\
		n_3y_3 + n_2 z_3 - az_1   & = & 0\\
		n_3z_3 - n_3 x_1 -n_3 y_2 & = & 0\\
		n_2x_2 + n_1 y_1  -a z_3  & = & 0\\
		n_2y_2 -n_2x_1- n_2 z_3   & = & 0\\
		n_2z_2 + n_3 y_3 +az_1 & = & 0\\
		n_1x_1 - n_1 y_2 - n_1z_3 & = & 0\\
		n_1y_1 + n_2 x_2 +a z_3 & = & 0\\
		-n_2x_2 + n_3 x_3 -ay_2 & = & 0.
	\end{array}
	\right.
\end{equation}
Thus, taking values to $a$, $n_1$, $n_2$ and $n_3$ in Table 1 in \cite{biggs} we can find derivations in each class of connected, simply connected Lie groups of dimension 3. 

\subsubsection{Type $3 \fg_1$ (Abelian Groups)}
	In this class the simply connected Lie group $G$ with Lie algebra $\fg \cong 3 \fg_1$ is isomorphic to $\R^3$. Lie brackets (\ref{colchetes}) assume values $a=n_1 = n_2 = n_3 = 0$. Trivially, by linear system (\ref{sys1}), the matrix of any derivation is written as
	\[
		\DC = 
		\left(
		\begin{array}{c c c}
			x_1 & x_2 & x_3 \\
			y_1 & y_2 & y_3 \\
			z_1 & z_3 & z_3 
		\end{array}
		\right).
	\]
	The characteristic polynomial of $\DC$ is $p(\lambda) =-\lambda ^3+ \operatorname{tr}(\DC)\lambda^2+ A \lambda+ \det(\DC)$ where $A = x_2y_1 -x_1y_2 + x_3z_1 -x_1z_3 + y_3z_2 -y_2z_3$. By Theorem \ref{teo1}, we want to work with derivation with one null eigenvalue and two conjugate complex eigenvalues. It is direct that $\det(\DC) = 0$. Then $p(\lambda) =\lambda(-\lambda^2+ \operatorname{tr}(\DC)\lambda + A)$ where $A = x_2y_1 -x_1y_2 + x_3z_1 -x_1z_3 + y_3z_2 -y_2z_3$ . Now consider the polynomial $q(\lambda) = - \lambda^2 + \operatorname{tr}(\DC) \lambda + A$. A simple account shows that the roots are $\lambda = (\operatorname{tr}{(\DC)} \pm \sqrt{\operatorname{tr}(\DC)^2 + 4A })/2$. By Theorem \ref{teo1} again, we want that $\operatorname{tr}(\DC) = 0$, whit implies that $\lambda_1 = \sqrt{A}$ and $\lambda_2 = -\sqrt{A}$. Summarizing, we have the conditions to the linear flow has periodic orbits. 
	\begin{proposition}
		Let $\varphi_t$ be a linear flow on the simply connected with Lie algebra of type $\fg_{3,6}$. Then $\varphi_t$ has periodic orbit if $\operatorname{tr}(\DC) = \det(\DC) = 0$ and $x_2y_1 -y_1x_2 + x_3z_1 - x_1z_3 + y_3z_2 - y_2z_3<0$.
	\end{proposition}

	Now let $X$ be a invariant vector field. Since $3 \fg_1$ is abelian, it follows that $\ad(X) = 0$. Consequently, $\DC = - \ad(X) = 0$. From Theorem \ref{teo2} we conclude that the invariant flow $\exp(tX)$ is not periodic. It means that the existence of periodic orbits for linear flow does not imply the existence of  periodic orbits for invariant flow. 
	
\subsection{Type $\fg_{2,1} \oplus \fg_1$}
	If a Lie algebra $\fg \cong \fg_{2,1} \oplus \fg_1$, then the semisimple Lie group $G$ is isomorphic to $Aff(\R)_0 \times \R$. In this case, the Lie bracket (\ref{colchetes}) is characterized by $a = 1$, $n_1 = 1$, $n_2 = -1$, and $n_3 = 0$. From linear system (\ref{sys1}) it follows that the matrix of any derivation is 
	\[
		\DC = 
		\left(
		\begin{array}{c c c}
			x_1 & x_2 & x_3 \\
			x_2 & x_1 & y_3 \\
			0 & 0 & 0 
		\end{array}
		\right),
	\]
	and its eigenvalues are $\left\{0,x_1-x_2,x_1+x_2\right\}$. 
	\begin{proposition}
		Any linear flow on a Lie group with Lie algebra of type $\fg_{2,1} \oplus \fg_1$ do not have periodic orbits.
	\end{proposition}
%	\[
%		\DC = 
%		\left(
%		\begin{array}{c c c}
%			 0   & x_2 & x_3 \\
%			x_2  & y_2 & y_2 \\
%			-x_3 & y_2 & y_2 
%		\end{array}
%		\right)
%	\]
%	and its characteristic polynomial of $\DC$ is $p(\lambda) = -\lambda ^3 +2 y_2 \lambda^2 + (x_2^2-y_2^2)\lambda + (y_2^2 -x_2^2)d$. It is simple to see that Girard equations of $p(\lambda)$ are
%	\begin{eqnarray*}
%		r_1+r_2 +r_3 & = & 2y_2 \\
%		r_1\cdot r_2 + r_1 \cdot r_3 + r_2 \cdot r_3 & = & x_3^2 -x_3^2 \\
%		r_1 \cdot r_2 \cdot r_3 & = & (x_3^2 - x_2^2) \cdot y_2.
%	\end{eqnarray*}
%	Suppose that a linear flow has periodic orbit, from Theorem \ref{teo1} we must have $r_1 =0$. It implies that $(x_3^2 - x_2^2) \cdot y_2 = 0$. If would $y_2 \neq 0$, then $x_3^2 -x_2^2 = 0$. It entails that $r_2 \cdot r_3 = 0$. So $r_2 = 0$ or $r_3 = 0$. Hence, for example, we have $r_1 = r_2 = 0 $ and $r_3 = 2y_2$, which is a contradiction by Corollary \ref{correaleigenvalue}.  So $y_2 =0$. It implies that
%	\[
%		\DC = 
%		\left(
%		\begin{array}{c c c}
%			 0   & x_2 & x_3 \\
%			x_2  & 0  & 0 \\
%			-x_3 & 0  & 0 
%		\end{array}
%		\right)
%	\]
%	Eigenvalues of derivation above are $\left\{0,-\sqrt{b^2-c^2},\sqrt{b^2-c^2}\right\}$. Now, a characterization of periodic orbits for linear flow is obtained from Theorem \ref{teo1}. 
%	\begin{proposition}
%		Let $\varphi_t$ be a linear flow on a Lie group with Lie algebra of type $\fg_{2,1} \oplus \fg_1$. Then $\varphi_t$ has periodic orbit if and only if $y_2= 0$ and $|x_2|<|x_3|$.
%	\end{proposition}
	
\subsubsection{Type $\fg_{3,1}$}
	The simply connected, matrix group with Lie algebra of type $\fg_{3,1}$ is isomorphic to Heisenberg group $H_3$. In this case, we adopt $a=0$, $n_1 = 1$, $n_2 = 0$, and $n_3 = 0$. From linear system (\ref{sys1}) we see that the matrix of any derivation is written as 
	\[
		\DC = 
		\left(
		\begin{array}{c c c}
			 y_2+z_3 & x_2 & x_3 \\
			 0  & y_2 & y_3 \\
			 0 & z_2 & z_3 
		\end{array}
		\right).
	\]
	The eigenvalues of derivation $\DC$ are  
	\[
		\begin{array}{cc}
		\left\{y_2+z_3, \right. &\frac{1}{2} \left(y_2+z_3 -\sqrt{(y_2- z_3)^2+4 x_3z_2}\right) ,\\
		    & \left. \frac{1}{2} \left(y_2+z_3+\sqrt{(y_2- z_3)^2+4 x_3z_2}\right)\right\}. 
		\end{array}
	\]
%	By Theorem \ref{teo1}, if $\varphi_t$ has periodic orbits we must have $y_2+ z_3= 0$ and $(y_2-z_3)^2 + 4x_3z_2<0$. It implies that derivation can be written as 
%	\begin{equation}\label{derH}
%		\DC = 
%		\left(
%		\begin{array}{c c c}
%			 0  & x_2 & x_3 \\
%			 0  & y_2 & y_3 \\
%			 0 & z_2 & -y_2 
%		\end{array}
%		\right).
%	\end{equation}
	
	\begin{proposition}
		Let $\varphi_t$ be a linear flow on a simply connected Lie group with Lie algebra of type $\fg_{3,1}$. Then $\varphi_t$ has periodic orbit if $y_2+z_3 = 0$ and  and $(y_2-z_3)^2+4x_3z_2 <0 $.
\end{proposition}

\subsection{Type $\fg_{3,2}$}
	The simply connected Lie group $G$ with Lie algebra $\fg \cong \fg_{3,2}$ is isomorphic to $G_{3,2}$. The Lie bracket is given by $a=1$, $n_1 = 1$, $n_2 = 0$, and $n_3 = 0$. It follows, by linear system (\ref{sys1}), that the matrix of any derivation is written as
	\[
		\DC = 
		\left(
		\begin{array}{c c c}
			 0 & x_2 & x_3 \\
			 0 & 0   & y_3 \\
			 0 & 0   & 0 
		\end{array}
		\right),
	\]
	and its eigenvalues are $\{0,0,0\}$.
	\begin{proposition}
		Let $\varphi_t$ be a linear flow on a simply connected Lie group with Lie algebra of type $\fg_{3,2}$. Then $\varphi_t$ do not have periodic orbits.
	\end{proposition}
	
\subsection{Type $\fg_{3,3}$}
	The simply connected Lie group $G$ with Lie algebra of $\fg \cong \fg_{3,3}$ is isomorphic to $G_{3,3}$ and its Lie bracket is given by $a = 1$, $n_1 = 0$, $n_2 = 0$, and $n_3 = 0$. From linear system (\ref{sys1}) we deuce that the matrix of any derivation is given by
	\[
		\DC = 
		\left(
		\begin{array}{c c c}
			 x_1 & x_2 & x_3 \\
			 y_1 & y_2 & y_3 \\
			 0 & 0   & 0 
		\end{array}
		\right),
	\]
	and its eigenvalues are 
	\[
		\left\{0,\frac{1}{2} \left(x_1+y_2-\sqrt{(x_2-y_2)^2+4 x_2 y_1}\right),\frac{1}{2} \left(x_1+y_2+\sqrt{(x_2-y_2)^2+4 x_2 y_1}\right)\right\}.
	\] 
%	By Theorem \ref{teo1}, if $\varphi_t$ has periodic orbits, we must have $x_1+y_2= 0$ and $(x_2-y_2)^2+4 x_2 y_1<0$. It implies that derivation can be written as 
%	\begin{equation}\label{der33}
%		\DC = 
%		\left(
%		\begin{array}{c c c}
%			 -y_2 & x_2 & x_3 \\
%			 y_1 & y_2 & y_3 \\
%			 0 & 0   & 0 
%%		\end{array}
%		\right),
%	\end{equation}
	\begin{proposition}
		Let $\varphi_t$ be a linear flow on a simply connected Lie group with Lie algebra of type $\fg_{3,3}$. Then $\varphi_t$ has periodic orbit if $x_1+y_2= 0$ and $(x_2-y_2)^2+4 x_2 y_1<0$.
\end{proposition}
	
\subsection{Type $\fg^0_{3,4}$}
	In this class, any simply connected Lie group $G$ with Lie algebra $\fg \cong \fg^0_{3,4}$ is isomorphic to $SE(1,1)$. Furthermore, Lie bracket is obtained by $a = 0$, $n_1 = 1$, $n_2 = -1$, and $n_3 = 0$. From linear system (\ref{sys1}) we see that the matrix of any derivation is 
	\[
		\DC = 
		\left(
		\begin{array}{c c c}
			x_1 & x_2& x_3 \\
			x_2 & x_1& y_3 \\
			0   & 0  & 0 
		\end{array}
		\right),
	\]
	and its eigenvalues are $\left\{0,x_1-x_2,x_1+x_2\right\}$. 
	\begin{proposition}
		Any linear flow on a Lie group with Lie algebra of type $\fg^0_{3,4}$ do not have periodic orbits.
	\end{proposition}
	
\subsection{Type $\fg^a_{3,4}$}
	Here, we have a family of Lie algebra of type $\fg^a_{3,4}$ given by conditions $a> 0$ and $a\neq 1$, $n_1 = 1$, $n_2 = -1$, and $n_3 = 0$. Furthermore, a simply connected Lie group $G$ with Lie algebra $\fg^a_{3,4}$ is isomorphic to $G^{a}_{3,4}$. By linear system (\ref{sys1}), the matrix of any derivation is written as
	\[
		\DC = 
		\left(
		\begin{array}{c c c}
			 -y_2 & ay_2 & x_3 \\
			 ay_2 & y_2 & y_3 \\
			 0 & 0 & 0 
		\end{array}
		\right),
	\]
	which yields as eigenvalues $\left\{0,-\sqrt{(1+a)y_2^2},\sqrt{(1+a)y_2^2}\right\}$.  
	\begin{proposition}
		For $a>0$ and $a\neq 1$, any linear flow on a simply connected Lie group with Lie algebra of type $\fg^a_{3,4}$ do not have periodic orbits.
	\end{proposition}

%\subsubsection{Type $\fg^0_{3,5}$ (The Euclidean Group and its Covers)}
%	A Lie group with Lie algebra of type $\fg^0_{3,5}$ is isomorphic to the Euclidean group $SE(2)$, the n-fold covering $SE_n(2)$ of $SE_1(2) = SE (2)$, or the universal covering group $\tilde{SE}(2)$. Lie algebra $\fg^0_{3,5}$ is given by Lie bracket with $a=0$, $n_1 = 1$, $n_2 = 1$ and $n_3 = 0$. It follows that the matrix of any derivation is 
%	\[
%		\DC = 
%		\left(
%		\begin{array}{c c c}
%			x_1 & 0  & x_3 \\
%			0   & x_1& y_3 \\
%			0   & 0  & 0 
%		\end{array}
%		\right)
%	\]
%	and its eigenvalues are $\{0,x_1,x_1\}$. 
%	\begin{proposition}
%		Any linear flow on a Lie group with Lie algebra of type $\fg^0_{3,5}$ do not have periodic orbits.
%	\end{proposition}

\subsection{Type $\fg^a_{3,5}$}
	A family of Lie algebra $\fg^a_{3,5}$ is characterize by $a> 0$, $n_1 = 1$, $n_2 = 1$, and $n_3 = 0$. In this case, simply connected Lie groups with Lie algebra are isomorphic to $G^a_{3,5}$. 
	Solving linear system (\ref{sys1}) we can write the matrix of any derivation as 
	\[
		\DC = 
		\left(
		\begin{array}{c c c}
			 y_2 & -ay_2 & x_3 \\
			 ay_2 & -ay_2 & y_3 \\
			 0 & 0 & 0 
		\end{array}
		\right),
	\]
	which yields as eigenvalues 
	\begin{eqnarray*}
		\{0, &\frac{1}{2} \left((-a-1)y_2-\sqrt{-2 ay_2 y_2-3 ay_2^2+y_2^2}\right), \\
							& \frac{1}{2} \left((-a-1)y_2+ \sqrt{-2 ay_2 y_2-3ay_2^2+y_2^2}\right)\}.
	\end{eqnarray*}
	By Theorem \ref{teo1}, we must have $ (-a-1)y_2 = 0$ for the linear flow $\varphi_t$ to have periodic orbits. Since $a >0$, it implies that $-a-1 \neq 0$. We thus get $y_2 = 0$.
%	ii) $a =1$\\
%	From Theorem \ref{teo1} we can written the matrix of any derivation as 
%	\[
%		\DC = 
%		\left(
%		\begin{array}{c c c}
%			 -x_2 & x_2 & x_3 \\
%			 -x_2 & -x_2 & y_3 \\
%			 0 & 0 & 0 
%		\end{array}
%		\right)
%	\]
%	and its values are $\left\{0,-x_2-x_2i ,-x_2+x_2i\right\}$. 
%	By Theorem \ref{teo1}, if $\varphi_t$ has periodic orbit, we must have $ x_2 = 0$.
	
%	In resume, both cases gives the same result. 
	
	\begin{proposition}
		Any linear flow on a simply connected Lie group with Lie algebra of type $\fg^a_{3,5}$ for some $a>0$ do not have periodic orbits.
	\end{proposition}

\end{document}